\documentclass[a4paper,leqno,11pt]{amsart}
\usepackage[
top=30truemm,
bottom=30truemm,
left=30truemm,
right=30truemm]{geometry}
\usepackage{amsfonts, amssymb, amsmath, amsthm, mathrsfs, amscd, enumerate, ascmac, fancyhdr, caption, hyperref, framed,ulem,subfigure}
\usepackage{bbm}%mathbbm{1}
\usepackage{eucal}%fancy \mathcal
\usepackage{tikz}
\usetikzlibrary{shapes.geometric}
\usetikzlibrary{arrows.meta,arrows}
\usepackage{graphicx} % Required for inserting images
\usepackage{pgf, tikz}
\usepackage{pgfplots}
\pgfplotsset{compat=1.17}
\usepackage{float}
\usepackage{pdfpages}

\usepackage{changes}

\usepackage[framemethod=tikz]{mdframed}
\mdfsetup{linecolor=blue,backgroundcolor=gray!10,roundcorner=10pt,leftmargin=2pt,innerleftmargin=2pt,innerrightmargin=2pt}
\usepackage{url}

\usepackage{xcolor}
\usepackage{tikz}
\usetikzlibrary{positioning,intersections, calc, arrows,decorations.markings}

\newtheorem{theorem}{Theorem}[section]
\newtheorem{lemma}[theorem]{Lemma}
\newtheorem{proposition}[theorem]{Proposition}
\newtheorem*{theorem*}{Theorem}

\newtheorem*{conjecture*}{Conjecture}

\newtheorem{corollary}[theorem]{Corollary}

\newtheorem*{claim*}{Claim}
\theoremstyle{definition}
\newtheorem{definition}[theorem]{Definition}

\newtheorem*{goal*}{Goal}

\numberwithin{equation}{section}
\numberwithin{theorem}{section}

 \DeclareMathOperator{\dist}{dist}

\newtheorem*{notation}{Notation}

\newtheorem{remark}[theorem]{Remark}

\newcommand{\BL}{\operatorname{BL}}

\newcommand{\FBL}{\operatorname{\widehat{BL}}}

\def\reals{{\mathbb R}}
\def\complex{{\mathbb C}}
\def\naturals{{\mathbb N}}
\def\scripts{{\mathcal{S}}}
\def\scriptt{{\mathcal{T}}}
\def\bff{{\mathbf{f}}}
\def\bg{{\mathbf{g}}}

\def\Bp{(\mathbf{B},\mathbf{p})} 
\def\BLBp{\BL({\mathbf{B},\mathbf{p}})}

\def\nullspace{\operatorname{nullspace}}
\def\bB{\mathbf{B}}
\def\bA{\mathbf{A}}
\def\bAstar{\mathbf{A}_*}
\def\bp{\mathbf{p}}
\def\ba{\mathbf{a}}
\def\bv{\mathbf{v}}
\def\barx{\overline{x}}
\def\eps{\varepsilon} 
\def\symm{\operatorname{Symm}}
\def\trace{\operatorname{trace}}

\newcommand{\norm}[1]{ \|  #1 \|}
 \newcommand{\blbp}[1]{\BL({\mathbf{B},\mathbf{p}}; #1 )}
\title[Stability for Brascamp--Lieb inequalities]
{Remarks on Gaussian-stability 
\\ for Brascamp--Lieb Inequalities}
\author{Jonathan Bennett and Michael Christ}\thanks{The first author was supported by the Engineering and Physical Sciences Research Council Grant EP/W032880/1. We thank the Mathematisches Forschungsinstitut Oberwolfach, and specifically the organizers of the 2025 workshop ``Real Analysis, Harmonic Analysis, and Applications", where this work was initiated.}
\address{Jonathan Bennett, School of Mathematics, University of Birmingham, Edgbaston, B15 2TT, UK}
\email{j.bennett@bham.ac.uk}
\address{Michael Christ, Department of Mathematics, University of California, Berkeley, CA 94720-3840}
\email{mchrist@berkeley.edu}
\date{January 5, 2026}
\begin{document}

\begin{abstract}
We establish a stable form of the general Euclidean Brascamp--Lieb inequality in all cases 
in which the Lebesgue exponents are strictly between $1$ and $2$, 
asserting that all near-extremizers are nearly Gaussian.
\end{abstract}
\maketitle

\section{introduction}
The Brascamp--Lieb inequality provides Lebesgue space bounds for a broad class of multilinear forms that arise frequently in mathematics, generalizing a variety of well-known functional inequalities, such as the H\"older, Loomis--Whitney and Young convolution inequalities. It takes the form
\begin{equation}\label{eq:BLlinearform}
\int_{\mathbb{R}^d}\prod_{j=1}^mf_j(B_jx)
	 \, \mathrm{d}x\leq\BL(\textbf{B},\textbf{p})\prod_{j=1}^m\|f_j\|_{L^{p_j}(\mathbb{R}^{d_j})},
\end{equation}
where the mappings $B_j:\mathbb{R}^d\rightarrow\mathbb{R}^{d_j}$ are linear surjections, 
the quantities $1\leq p_j\leq \infty$ are Lebesgue exponents, 
the inputs $f_j$ are arbitrary nonnegative real-valued functions on $\mathbb{R}^{d_j}$ for each $1\leq j\leq m$,  
and $\BL(\textbf{B},\textbf{p})$ is the infimum of all elements of $[0,\infty]$
for which inequality holds for all $(f_j) = (f_1,\dots,f_m)$. 
Following \cite{BCCT}, the pair of $m$-tuples $(\textbf{B},\textbf{p})=((B_j),(p_j))$ is referred to as the \textit{Brascamp--Lieb datum}, and the optimal constant $0<\BL(\mathbf{B},\mathbf{p})\leq\infty$ is referred to as the \textit{Brascamp--Lieb constant}. It is well-known that $\BL(\mathbf{B},\mathbf{p})$ is finite if and only if the scaling condition
\begin{equation}\label{BLscaling} 
d = \sum_{j=1}^m 
	p_j^{-1}d_j
\end{equation} holds and
every linear subspace $V$ of $\reals^d$ is subcritical in the sense that
\begin{equation}\label{transversality}
	\dim(V) \leq \sum_j p_j^{-1}  \dim(B_j(V)).
\end{equation}
We refer to \cite{BCCT, BCCT2} and the references there for this and other structural results relating to $\BL(\mathbf{B},\mathbf{p})$. In particular, a foundational
theorem of Lieb \cite{L} asserts
that the Brascamp--Lieb constant is always saturated in the supremum sense
by centered Gaussian inputs $\mathbf{f}=(f_j)$, and there is a good understanding of the data for which Gaussian extremizers exist.
In the presence of such information it is often of interest to obtain stability statements, whereby inputs that exhibit near-equality are identified as being near to (in this case) Gaussians with respect to a suitable metric. 
Reaching such conclusions for \eqref{eq:BLlinearform} is the main objective of this paper. 

In what follows we denote by $\mathfrak{G}_{\mathbb{R}_+}(d)$ the set of all positive (uncentered) Gaussians on $\mathbb{R}^d$, that is, functions $f$ of the form \begin{equation}\label{gaussdef} f(x)=ce^{-Q(x) +v\cdot x}\end{equation} where $Q$ is a positive definite real quadratic form on $\mathbb{R}^d$, $v\in\mathbb{R}^d$ and $c\in(0,\infty)$.
We refer to elements of $\mathfrak{G}_{\mathbb{R}_+}(d)$ as \textit{positive Gaussians}, and for $f\in L^p(\reals^d)$ define
\begin{equation} \dist_p(f,\mathfrak{G}_{\reals_+}(d)) = \inf_{g\in \mathfrak{G}_{\reals_+}(d)}
\norm{f-g}_{L^p}.\end{equation} 

Our approach is in part Fourier-analytic and 
leads naturally
to stability results in the more general setting of complex-valued functions
$f_j$. 
Whenever $\BLBp<\infty$, the integral
$\int_{\reals^d} \prod_j |f_j(B_jx)|\mathrm{d}x$
is finite for any functions $f_j\in L^{p_j}(\reals^{d_j})$,
and therefore 
$\int_{\reals^d} \prod_j f_j(B_jx)\mathrm{d}x$ 
is well-defined and satisfies 
\begin{equation}\label{eq:BLlinearformcomplex}
\Big|\int_{\mathbb{R}^d}\prod_{j=1}^mf_j(B_jx)
	 \, \mathrm{d}x\Big|\leq\BL(\textbf{B},\textbf{p})\prod_{j=1}^m\|f_j\|_{L^{p_j}(\mathbb{R}^{d_j})}.
\end{equation}
We denote by $\mathfrak{G}_{\mathbb{C}}(d)$ the set of functions on $\mathbb{R}^d$  of the form \eqref{gaussdef} 
where $Q$ is a positive definite real quadratic form on $\mathbb{R}^d$, $v\in\mathbb{C}^d$, 
and $0\ne c\in\mathbb{C}$. 
We refer to elements of $\mathfrak{G}_{\mathbb{C}}(d)$ as 
\textit{complex Gaussians}, and define 
$\dist_p(f,\mathfrak{G}_{\complex}(d))$ in the corresponding way.
\begin{notation}
Assuming finiteness of $\BLBp$, we define $\blbp{\bff}$ by
\begin{equation}
\blbp{\bff} 
= \frac{\Big| \int_{\reals^d} \prod_{j=1}^m f_j\circ B_j\Big|}
{\prod_{j=1}^m \norm{f_j}_{L^{p_j}}}
\end{equation}
for any tuple $\bff = (f_j: 1\le j\le m)$
satisfying $f_j\in L^{p_j}$ and $\norm{f_j}_{L^{p_j}}>0$ for every $j$.
\end{notation}
Thus $\BLBp = \sup_\bff \blbp{\bff}$
with the supremum taken over all such complex-valued tuples $\bff$. 
Further, Lieb's theorem is the assertion that for every $\delta>0$ there exists a tuple $\bff$ of (positive centered) gaussians for which $\blbp{\bff} \geq (1-\delta)\BLBp$. 
\begin{definition}[Gaussian-stability]\label{def:Gaussian-stability}
    A Brascamp--Lieb datum $(\mathbf{B},\mathbf{p})$ is \textit{positive Gaussian-stable} if $\BL(\mathbf{B},\mathbf{p})$ is finite, and given any $\eps>0$ there exists a $\delta>0$ such that 
\begin{equation} \label{GS}
\dist_{p_j}(f_j,\mathfrak{G}_{\mathbb{R}_+}(d_j)) \le \eps\|f_j\|_{L^{p_j}(\mathbb{R}^{d_j})}
\end{equation}
for each $1\leq j\leq m$
whenever the nonnegative functions $f_j$ have positive, finite norms and
satisfy
    \begin{equation}\label{eq:nearBL}
	    \blbp{\bff} \ge (1-\delta)\BLBp.
    \end{equation}
    Similarly, a datum $(\mathbf{B},\mathbf{p})$ is \textit{complex Gaussian-stable} if $\BL(\mathbf{B},\mathbf{p})$ is finite, and given any $\eps>0$ there exists a $\delta>0$ such that 
\begin{equation} \label{GSreal}
\dist_{p_j}(f_j,\mathfrak{G}_{\mathbb{C}}(d_j)) \le \eps\|f_j\|_{L^{p_j}(\mathbb{R}^{d_j})}
\end{equation}
for each $1\leq j\leq m$
whenever the complex-valued functions $f_j$ have positive, finite norms and satisfy
\eqref{eq:nearBL}.
\end{definition}

The reader may quickly verify that
\begin{equation}\label{complextoreal} \dist_{p_j}(f_j,\mathfrak{G}_{\mathbb{C}}(d_j))=\dist_{p_j}(f_j,\mathfrak{G}_{\mathbb{R}_+}(d_j))
\end{equation}  
whenever $f_j$ is nonnegative, and so complex Gaussian-stability implies positive Gaussian-stability.

The problem of characterizing positive Gaussian-stable Brascamp--Lieb data was raised recently in 
\cite[{Question 10.2}]{BT} 
in the estimation of adjoint Brascamp--Lieb constants.
It was shown in \cite{C:stable Y} that 
complex Gaussian-stability holds for
non-endpoint cases of Young's convolution inequality, 
but little had been known 
concerning more general data.
This note establishes Gaussian-stability in substantial --- though far from
complete --- generality.
In fact we will prove the following quantitative stability statement 
in the general setting of complex-valued functions.
\begin{theorem}[A sharpened Brascamp--Lieb inequality below $L^2$]\label{thm:sharpenedBL}
	If $\BLBp$ is finite and $1<p_j<2$ for all $j$ 
then there are constants $c_j=c(p_j,d_j)>0$ such that
    \begin{equation}\label{eq:BLlinearformsharpened}
   \blbp{\bff} \le \BLBp 
\prod_{j=1}^m \Big[1-c_j\Big(\frac{\dist_{p_j}(f_j,\mathfrak{G}_{\mathbb{\mathbb{C}
}}(d_j))}{\|f_j\|_{L^{p_j}(\mathbb{R}^{d_j})}}\Big)^2\Big]
\end{equation}
for all tuples $\bff = (f_j: 1\le j\le m)$ of
nonzero complex-valued $f_j\in L^{p_j}(\mathbb{R}^{d_j})$. 
\end{theorem}
In particular, if the hypotheses of Theorem~\ref{thm:sharpenedBL}
are satisfied, if each $f_j$ has strictly positive norm, and if 
$\bff$ realizes equality 
in the Brascamp-Lieb inequality \eqref{eq:BLlinearformcomplex}, 
then 
$\dist_{p_j}(f_j, \mathfrak{G}_{\mathbb{\mathbb{C}}}(d_j)) =0$
for each $j$ and therefore since  $\mathfrak{G}_{\mathbb{\mathbb{C}}}(d_j)$
is a closed subspace of $L^{p_j}\setminus\{0\}$,
each $f_j$ agrees almost everywhere with a complex Gaussian.

Specializing to nonnegative functions in the statement of Theorem \ref{thm:sharpenedBL}, 
where \eqref{complextoreal} holds, the conclusion \eqref{eq:BLlinearformsharpened} becomes
\begin{equation}\label{eq:BLlinearformsharpenedreal}
\blbp{\bff} \le
\BL(\mathbf{B},\mathbf{p})\prod_{j=1}^m\Big[1-c_j\Big(\frac{\dist_{p_j}(f_j,\mathfrak{G}_{\mathbb{\mathbb{R}_+ }}(d_j))}{\|f_j\|_{L^{p_j}(\mathbb{R}^{d_j})}}\Big)^2\Big].
\end{equation}
\begin{corollary}\label{cor:GS}
If $\BL(\mathbf{B},\mathbf{p})$ is finite and $1<p_j<2$ for every $1\leq j\leq m$, then  $(\mathbf{B},\mathbf{p})$ is complex Gaussian-stable (and thus positive Gaussian-stable). Moreover, in this case  \eqref{GS} (and thus \eqref{GSreal}) hold with $\delta=C\eps^2$ 
for some constant $C>0$ depending only on the exponents $p_j$ and dimensions $d_j$.
\end{corollary}

It is natural to seek a stronger conclusion, in which closeness of the individual 
$f_j$ to Gaussians is replaced by closeness of the $m$-tuple $\bff = (f_j: 1\le j\le m)$
to a tuple of Gaussians that realizes the optimal constant in the inequality.
The proof of our next theorem will illustrate
how the component-wise conclusion of Corollary~\ref{cor:GS}
can sometimes be built upon to obtain a conclusion of that type.

We say that a tuple $\bg$ realizes the optimal constant $\BL(\mathbf{B},\mathbf{p})$ 
in a Brascamp-Lieb inequality if $0<\norm{g_j}_{p_j}<\infty$ for every index $j$ and
\begin{equation}
\Big|\int_{\mathbb{R}^n}\prod_{j=1}^m g_j(B_jx)
 \,\mathrm{d}x\Big| 
= \BL(\mathbf{B},\mathbf{p})\prod_{j=1}^m\|g_j\|_{L^{p_j}(\mathbb{R}^{d_j})}, 
\end{equation}
that is, if $\blbp{\bg} = \BLBp$.

Recall from \cite{BCCT} that $(\mathbf{B},\mathbf{p})$ is said to be \textit{simple}
if \eqref{BLscaling} holds and
every nonzero proper subspace $V$ of $\reals^d$ is strictly subcritical, that is, satisfies 
\eqref{transversality} with strict inequality.
Thus any
simple Brascamp--Lieb datum has finite Brascamp--Lieb constant. 
A key feature of simple data, established in \cite{BCCT},
is that positive centered
Gaussian extremizers exist and are unique, up to the scalings permitted by homogeneity and \eqref{BLscaling}. 

\begin{theorem}\label{thm:main_variant}
Suppose that $(\mathbf{B},\mathbf{p})$ is simple
and that $p_j\in(1,2)$ for every index $j\in\{1,2,\dots,m\}$.
For every $\eps>0$ there exists $\delta>0$ such that 
for any tuple $\bff = (f_j: 1\le j\le m)$ of nonnegative measurable functions satisfying
$\blbp{\bff} \ge (1-\delta)\BLBp$,
there exists 
a tuple $\bg = (g_j: 1\le j\le m)$ of nonnegative Gaussians that 
realizes the optimal constant $\BL(\mathbf{B},\mathbf{p})$ and satisfies
\[ \norm{f_j-g_j}_{L^{p_j}}
\le \eps \|f_j\|_{L^{p_j}(\mathbb{R}^{d_j})}\ \forall\,j\in\{1,2,\dots,m\}.  \]
\end{theorem}

The hypothesis that $(\mathbf{B},\mathbf{p})$ is simple is not necessary. 
We hope to discuss more general $(\mathbf{B},\mathbf{p})$ in a sequel.

\begin{remark}\label{e.g: rank 1} 
 Loosely speaking, data for which the dimensions $d_j$ are relatively small and the 
degree of multilinearity $m$ is not too high often satisfy the hypothesis $1<p_j<2$, as suggested by the scaling condition \eqref{BLscaling}.
For example, if $d_1=\cdots=d_m=1$ and $p_1=\cdots =p_m=p:=d^{-1}m$, then $\BL(\mathbf{B},\mathbf{p})$ is finite for generic choices of linear maps $B_1,\hdots, B_m$ whenever $m\geq d$, as can be deduced from H\"older's inequality together with Fubini's theorem.
The hypotheses of Theorem \ref{thm:sharpenedBL} are then satisfied whenever $d<m<2d$. The conclusion of Theorem \ref{thm:sharpenedBL} is easily seen to fail if either $m=d$ or $m=2d$. In the first case, which corresponds to $p=1$, the associated Brascamp--Lieb inequality amounts to Fubini's identity, for which  \eqref{eq:BLlinearformsharpened}
fails dramatically. In the second case, which corresponds to $p=2$, a counterexample is provided by a product of $d$ copies of the one-dimensional Cauchy--Schwarz inequality; see also Remarks \ref{remark:Holder} and \ref{remark:complex_case}, below.  \end{remark}

\begin{remark}\label{conventional}
Brascamp--Lieb inequalities are often equivalently
formulated for nonnegative functions as
    \begin{equation}\label{eq:BLlinearformcj}
\int_{\mathbb{R}^n}\prod_{j=1}^mf_j(B_jx)^{q_j}
	    \, \mathrm{d}x\leq\widetilde{\BL}(\textbf{B},\textbf{q})
 \prod_{j=1}^m\Big(\int_{\mathbb{R}^{d_j}}f_j\Big)^{q_j},
\end{equation}
with exponents $0\leq q_j\leq 1$.
This evidently transforms into \eqref{eq:BLlinearform}, with
$\widetilde{\BL}(\textbf{B},\textbf{q})
= \BL(\textbf{B},\textbf{p})$,
	on setting $q_j=1/p_j$ and replacing the (nonnegative) function $f_j$ by $f_j^{p_j}$ for each $j$. The corresponding notion of positive Gaussian-stability in this setting (see for example \cite[Definition 7.1]{BT}) then reads as follows: given $\eps>0$ there exists $\delta>0$ such that 
$$\dist_1(f_j,\mathfrak{G}_{\mathbb{R}}(d_j))
\le \eps\int_{\mathbb{R}^{d_j}}f_j$$ for each $j$
whenever 
$$
\int_{\mathbb{R}^d}\prod_{j=1}^mf_j(B_jx)^{q_j} \, \mathrm{d}x
	\ge (1-\delta)\, \widetilde{\BL}(\textbf{B},\textbf{q})
\prod_{j=1}^m\Big(\int_{\mathbb{R}^{d_j}}f_j\Big)^{q_j}.
$$
It is straightforward to verify that this notion of (real) Gaussian-stability 
coincides with that of Definition \ref{def:Gaussian-stability}, 
and so Theorem \ref{thm:sharpenedBL} establishes this stability whenever $\widetilde{\BL}(\mathbf{B},\mathbf{q})$ is finite and $\frac{1}{2}<q_j<1$ for all $j$. 
\end{remark}

\begin{remark} A variant of the above notion of real Gaussian-stability is raised in
Conjecture~2.2 of \cite{BFH}, where stability results are established for 
volume inequalities for $L^p$ zonoids.  \end{remark}

\begin{remark}
Related stability questions for functionals $(f_j)\mapsto \int \prod_j (f_j\circ B_j)$,
with the inputs $f_j$ restricted to indicator functions of sets
have been addressed for certain data $(B_j: 1\le j\le m)$ in several works, including
\cite{C:RSSH, C:RS, C:NEG, CI, CM, CN, EFKY}. 
\end{remark}

\section{Proof of Theorem \ref{thm:sharpenedBL}}
Theorem \ref{thm:sharpenedBL} is a direct consequence of the combination of a certain Fourier invariance property of Brascamp--Lieb constants with a stable form of the sharp Hausdorff--Young inequality. We begin with the former, 
which involves 
the \textit{Fourier--Brascamp--Lieb constant} $\FBL(\mathbf{B},\mathbf{p})$, 
defined to be
the optimal constant in the Fourier--Brascamp--Lieb inequality 
    \begin{equation}\label{eq:FBLlinearform}
\Big|\int_{\mathbb{R}^n}\prod_{j=1}^mf_j(B_jx)
	  \, \mathrm{d}x\Big|\leq\FBL(\mathbf{B},\mathbf{p})\prod_{j=1}^m\|\widehat{f}_j\|_{L^{p_j'}(\mathbb{R}^{d_j})}
   \text{  \ $\forall\,f_j\in\scripts_\complex(\reals^{d_j})$,}
\end{equation}
where $\scripts_\complex(\reals^n)$ denotes the class of all 
complex-valued Schwartz functions with domain $\reals^n$ 
and $p'$ denotes the exponent conjugate to $p$.
We refer to \cite{BBBCF} for the origins of this in the setting of general Brascamp--Lieb data, along with the more recent \cite{BJ,BC}.

Introduce the constants 
\[ \mathbf{A}_p=\Big(\frac{p^{1/p}}{p'^{1/p'}}\Big)^{1/2}.\]
\begin{theorem}[Fourier invariance of Brascamp--Lieb constants \cite{BBBCF}]\label{thm:FBL} For any Brascamp--Lieb datum $(\mathbf{B},\mathbf{p})$,
 $$\FBL(\mathbf{B},\mathbf{p})=
\BL(\mathbf{B},\mathbf{p}) \cdot \prod_{j=1}^m\mathbf{A}_{p_j}^{-d_j}. $$
\end{theorem}
For the purposes of establishing Theorem \ref{thm:sharpenedBL} the pertinent observation 
 is that if $p_j\leq 2$ for all $j$ then by Beckner's sharp Hausdorff--Young inequality
\begin{equation}\label{sharpHY}
    \|\widehat{f}\|_{L^{p'}(\mathbb{R}^d)}\leq \mathbf{A}_p^d\|f\|_{L^p(\mathbb{R}^d)},
\end{equation}
which holds for all $1\leq p\leq 2$  (\cite{B:Y1,B:Y2}), 
\eqref{eq:BLlinearform} (or equivalently \eqref{eq:BLlinearformcomplex}) may be strengthened to
\begin{equation}\label{eq:FBLlinearformstronger}
\Big|\int_{\mathbb{R}^n}\prod_{j=1}^mf_j(B_jx)
	\, \mathrm{d}x\Big|\leq\BL(\mathbf{B},\mathbf{p})\prod_{j=1}^m \mathbf{A}_{p_j}^{-d_j}\|\widehat{f}_j\|_{L^{p_j'}(\mathbb{R}^{d_j})}.
\end{equation}

We next review
the stable form of the sharp Hausdorff--Young inequality \eqref{sharpHY}. 
\begin{theorem}[Stable Hausdorff--Young Inequality \cite{C:stable HY}]\label{thm: stable HY}
For each $p\in(1,2)$  there exists $c=c(p,d)>0$ such that 
   $$
  \|\widehat{f}\|_{L^{p'}(\mathbb{R}^d)}\leq
\mathbf{A}_p^d \Big[ 1-c\Big(\frac{\dist_p(f,\mathfrak{G}_{\mathbb{C}}(d_j))}{\|f\|_{L^p(\mathbb{R}^d)}}\Big)^2\Big]\|f\|_{L^p(\mathbb{R}^d)}
   $$
   for all nonzero $f\in L^p(\mathbb{R}^d)$.
\end{theorem}
\begin{proof}[Proof of Theorem \ref{thm:sharpenedBL}]
Since $p_j\in (1,2)$, Theorem \ref{thm: stable HY} provides a constant $c_j=c_j(p_j,d_j)>0$ such that
$$
\|\widehat{f}_j\|_{L^{p_j'}(\mathbb{R}^d)}\leq
\mathbf{A}_{p_j}^{d_j}\Big[ 
1-c_j\Big(\frac{\dist_{p_j}(f_j,\mathfrak{G}_{\mathbb{C}}(d_j))}{\|f_j\|_{L^{p_j}(\mathbb{R}^{d_j})}}\Big)^2\Big]\|f_j\|_{L^{p_j}(\mathbb{R}^{d_j})}
$$
for each $1\leq j\leq m$. The theorem now follows from \eqref{eq:FBLlinearformstronger}.
\end{proof}

\begin{proof}[Proof of Corollary~\ref{cor:GS}]
Let $\bff$ be a tuple of $\complex$-valued functions 
satisfying $0<\norm{f_j}_{L^{p_j}}<\infty$.
Comparing the hypothesis
$\blbp{\bff} \ge (1-\delta) \BLBp$
of the corollary with the conclusion
\[ \blbp{\bff} \le \BLBp \prod_{j=1}^m
\Big[1-c_j\Big(\frac{\dist_{p_j}(f_j,\mathfrak{G}_{\mathbb{C}}(d_j))}{\|f_j\|_{L^{p_j}(\mathbb{R}^{d_j})}}\Big)^2\Big]
\]
of Theorem~\ref{thm:sharpenedBL} gives
\[  \sum_{j=1}^m
\Big(\frac{\dist_{p_j}(f_j,
\mathfrak{G}_{\mathbb{C}}(d_j))}{\|f_j\|_{L^{p_j}(\mathbb{R}^{d_j})}}\Big)^2
\le C\delta,
\]
which is a restatement of the announced conclusion with $\delta \asymp \eps^2$.
\end{proof}

\section{Remarks and examples}
Five remarks help to place our results in context,
and to delimit the scope of any potential extensions.
\begin{remark}
While the Brascamp--Lieb inequality \eqref{eq:BLlinearform} for nonnegative
functions $f_j$ is straightforwardly equivalent to its extension
\eqref{eq:BLlinearformcomplex}
to complex-valued functions in $L^{p_j}$,
results characterizing extremizers and near-extremizers for
nonnegative inputs do not imply corresponding results for complex-valued inputs.
For example,
for multilinear forms associated to Gowers--Host--Kra norms, 
nonnegative maximizing  tuples are Gaussian, while complex maximizing
tuples involve oscillatory factors $e^{iP}$
where $P$ are real-valued polynomials whose degrees
can be arbitrarily high when the degree of multilinearity becomes arbitrarily high.
A stability result for $\complex$-valued near-extremizers for Gowers--Host--Kra norms
has been obtained by Neuman \cite{AMN}.
In this regard see also Remark~\ref{remark:complex_case}.
\end{remark}

\begin{remark}\label{remark:Holder}
Gaussian-stability does not hold for exponents $p_j\in [2,\infty)$
without further hypotheses, even for nonnegative inputs. Indeed, for any $m\ge 2$ 
and any exponents $p_j\in(1,\infty)$ satisfying $\sum_{j=1}^m p_j^{-1}=1$,
there are nonnegative tuples $(f_j)$ that realize equality in H\"older's inequality
\[ \Bigl|\int_{\reals^d} \prod_{j=1}^m f_j(x) \,\mathrm{d}x\Bigr| \le \prod_{j=1}^m\norm{f_j}_{p_j}\]
with none of the functions $f_j$ close in $L^{p_j}$ norm to Gaussians.
On the other hand, Gaussian-stability does
sometimes hold for exponents in $[2,\infty)$. Indeed,
Gaussian-stability of the Brascamp--Lieb data corresponding to Young's convolution inequality was established in all non-endpoint cases in \cite{C:stable Y, C:stable Y2}. 

The elementary examples in Remark \ref{e.g: rank 1} further clarify that positive Gaussian-stability can fail if the condition $\mathbf{p}\in (1,2)^m$ imposed in the statement of Theorem \ref{thm:sharpenedBL} is replaced with either $\mathbf{p}\in [1,2)^m$ or $\mathbf{p}\in (1,2]^m$.
\end{remark}

\begin{remark}
The power $2$ to which the projective distance from $f_j$ to $\mathfrak{G}_{\mathbb{C}}$ is raised
in the statements of Theorem \ref{thm:sharpenedBL} and its corollary are best possible, 
in the sense of the next proposition. 
\begin{proposition}\label{prop:opt1}
There is no Gaussian-extremizable Brascamp--Lieb datum $(\mathbf{B},\mathbf{p})$ for which \eqref{eq:BLlinearformsharpenedreal}, and thus \eqref{eq:BLlinearformsharpened}, hold with a power of the (projective) distance below $2$.
\end{proposition}
For the particular case of Young's convolution inequality, 
this was shown in \cite{C:stable Y2}.
\begin{proof}
    Fix an extremizing tuple $(g_j)$ of Gaussians, a further tuple $(h_j)$ of nonzero compactly-supported smooth functions,
    and consider the one-parameter family $(f_{t,j})$ of (normalized) perturbations of $(g_j)$ given by $$f_{t,j} = (g_j+ th_j)/\|g_j+ th_j\|_{L^{p_j}(\mathbb{R}^{d_j})}.$$ It follows that 
    $$
    \frac{\dist_{p_j}(f_{t,j},\mathfrak{G}_{\mathbb{R}_+}(d_j))}{\|f_{t,j}\|_{L^{p_j}(\mathbb{R}^{d_j})}}\gtrsim |t|
    $$
    in a neighbourhood of $t=0$ for each $1\leq j\leq m$.
    Next we observe that the function
    $$u(t):=\int_{\mathbb{R}^d}\prod_{j=1}^mf_{t,j}(B_jx)\mathrm{d}x$$ is smooth and satisfies $u(0) = \BL(\mathbf{B},\mathbf{p})$ and $u'(0)=0$. Were \eqref{eq:BLlinearformsharpenedreal} to hold with power of the (projective) distance equal to $r$, then it would follow that 
    $$u(t) -u(0)-tu'(0) \leq -c\BL(\mathbf{B},\mathbf{p})|t|^r$$ for some constant $c>0$.
    However, since $u\in C^2(\mathbb{R})$, this must fail for $t$ small enough if $r<2$.
\end{proof}
\end{remark}

\begin{remark}
If some $p_k>2$ then \eqref{eq:BLlinearformsharpened} cannot
hold with $\dist_{p_k}(f_k,\mathfrak{G}_{\mathbb{C}})$ raised to the power $2$:
\begin{proposition}\label{prop:opt2}
Suppose that $(\mathbf{B},\mathbf{p})$ is Gaussian-extremizable. If $p_k>2$ for some $k$, then \eqref{eq:BLlinearformsharpened}, and thus \eqref{eq:BLlinearformsharpenedreal}, fails for any sequence of positive constants $c_j$.
\end{proposition}
For the special case of Young's convolution inequality, this was observed in \cite{C:stable Y}.
\begin{proof}[Proof of Proposition~\ref{prop:opt2}] 
By the linear-invariance of the projective distances in \eqref{eq:BLlinearformsharpened} and \eqref{eq:BLlinearformsharpenedreal}, combined with the fact that Gaussian-extremizable data is equivalent to geometric data (see \cite{BCCT}), it suffices to establish Proposition \ref{prop:opt2} for geometric data -- that is, for data satisfying $B_jB^*_j=I$ for each $j$, along with the frame condition
\begin{equation}\label{GBL}
	\sum_{j=1}^m p_j^{-1} B_j^*B_j=I.
\end{equation}
We recall that $\BL(\mathbf{B},\mathbf{p})=1$ for such data; see for example \cite{BCCT}.
    Suppose now that \eqref{eq:BLlinearformsharpenedreal} holds for some $p_1>2$, and let $g_j:\mathbb{R}^{d_j}\rightarrow\mathbb{R}_+$ be given by $g_j(x)=e^{-\pi|x|^2/p_j}$ for each $1\leq j\leq m$; these are conveniently chosen so that $\|g_j\|_{L^{p_j}(\mathbb{R}^{d_j})}=1$. As the datum $(\mathbf{B},\mathbf{p})$ is geometric, the tuple $(g_j)$ is an extremizer. Next, for a smooth compactly-supported function $\varphi:\mathbb{R}^{d_1}\rightarrow \mathbb{R}_+$, let $$f_1=g_1+\delta\varphi(x-t(\delta)v)$$ and $f_j=g_j$ whenever $1< j\leq m$; here $\delta>0$ is a small parameter, $v$ is a nonzero element of $\mathbb{R}^{d_1}$, and the function $t(\delta)\rightarrow\infty$ as $\delta\rightarrow 0$ in such a way that \begin{equation}\label{growth}\frac{t(\delta)}{\log(1/\delta)}\rightarrow \infty.\end{equation} Thanks to this growth it follows that 
    $$\|f_1\|_{L^{p_1}(\mathbb{R}^{d_1})}^{p_1}=1+O(\delta^{p_1}).$$ Further, by \eqref{GBL} we have
    $$\int_{\mathbb{R}^d}\prod_{j=1}^m f_j(B_jx)\mathrm{d}x=1+\delta\int_{\mathbb{R}^d}e^{-\pi Q(x)}\varphi(x-t(\delta)v)\mathrm{d}x,$$
	where $$Q(x)=\sum_{j=2}^m|B_jx|^2/p_j=\Big\langle\sum_{j=2}^m p_j^{-1} 
	B_j^*B_jx,x\Big\rangle.$$
    Since $p_1>1$ and $B_1^*B_1$ is a projection, it follows from \eqref{GBL} that $Q$ is positive-definite, and appealing once again to \eqref{growth} we conclude that
$$
\frac{\int_{\mathbb{R}^d}\prod_{j=1}^m f_j(B_jx)
\,\mathrm{d}x}{\prod_{j=1}^m\|f_j\|_{L^{p_j}(\mathbb{R}^{d_j})}}=1+O(\delta^{p_1}).
$$
This contradicts \eqref{eq:BLlinearformsharpenedreal} since $p_1>2$ and $$\frac{\dist_{p_1}(f_1,\mathfrak{G}_{\mathbb{R}_+}(d_1))}{\|f_1\|_{L^{p_1}(\mathbb{R}^{d_1})}}\gtrsim \|f_1-g_1\|_{L^{p_1}(\mathbb{R}^{d_1})}\gtrsim\delta.$$
\end{proof}
\end{remark}

\begin{remark} \label{remark:complex_case}
The general theory for complex-valued extremizers
tends to diverge from that for nonnegative extremizers
when the number $m$ of factors is excessively large.
To illustrate this, let $d\ge 2$ and
$m>d$. Let $v_1,\dots,v_m$ be elements of $\reals^d$,
any $d$ of which are linearly dependent.
Define $B_j:\reals^d\to\reals^1$ 
by $B_j(x) = \langle x,v_j\rangle$ for each index $j$.
Let $\bp = (p_j)$ with $p_j = m/d$ for all $j\in\{1,2,\dots,m\}$.
Then  $\Bp$ is simple, and so a positive Gaussian extremizer exists.

Indeed, the condition \eqref{transversality}  holds with strict inequality
for every nonzero proper subspace of $\reals^m$.
To verify this, consider any  subspace $V\subset\reals^d$ 
of dimension $k\in\{1,2,\dots,m-1\}$.
For any index $j$, $\dim(B_j(V))=1$ unless $v_j\in V^\perp$.
The linear independence hypothesis ensures that
the number of such indices $j$ 
does not exceed $\dim(V^\perp) = d-k$. 
Therefore \[ \sum_j p_j^{-1} \dim(B_j(V)) \ge dm^{-1}\cdot(m-(d-k)).\]
This quantity is $>k$ for $k<d$, since 
it equals $k$ when $k=d$ and its derivative with respect to $k$ equals $d/m<1$.

Fix a Gaussian extremizer $\bg = (g_j)$, with each $g_j > 0$. 
Let $a_j\in\reals$ be quantities to be chosen,
and for each $j$ define
\[ f_j(y) = g_j(y)\, e^{ia_j y^n} \ \text{ for all $y\in\reals^1$.} \]
Then $\BL(\mathbf{B},\bp; \bff) = \BLBp$ if and only if
$ \sum_j a_j \langle x,\,v_j\rangle^n=0\ \forall\,x\in\reals^d$. 

Now choose $n=2$.
The mapping \[\ba = (a_j: j\in\{1,2,\dots,m\})\mapsto
\sum_j a_j \langle x,\, v_j\rangle^2\]
is a real-linear mapping from a real vector space of dimension $m$
to the vector space of all real-valued homogeneous polynomials $P:\reals^d\to\reals$ 
of degree $2$ (regarding $0$ as a polynomial of degree $2$ for this discussion).
The latter vector space has dimension $d(d+1)/2$. 
Therefore if $m>d(d+1)/2$ then there exists $\ba\ne 0$ satisfying 
$\sum_j a_j \langle x,\, v_j\rangle^2=0\ \forall\,x\in\reals^d$. 
Then $\bff$ is an extremizer for the associated Brascamp--Lieb functional, but is not a 
tuple of complex Gaussians in the sense of the complex-valued case of Theorem~\ref{thm:sharpenedBL}.

When $d=2$ and $m=4$, $p_j = m/d=2$ for all $j$,
and this example then lies just on the boundary of the regime to 
which Theorem~\ref{thm:sharpenedBL} applies.
By forming direct sums, one obtains examples with $m=2d$ and $p_j=2$
for any $d\in 2\naturals$.

The same reasoning applies for any $n\ge 3$, as well,
but requires a correspondingly larger $m$ for each $d$.
\end{remark}

\section{Stability for tuples; proof of Theorem~\ref{thm:main_variant}}

By Theorem~\ref{thm:sharpenedBL}, it suffices to prove Theorem~\ref{thm:main_variant}
in the special case in which the functions $f_j$ are all positive Gaussians.
The next lemma asserts more: For positive Gaussian inputs, the stability conclusion holds 
for arbitrary exponents $p_j\in(1,\infty)$, rather than merely for exponents in $(1,2)$.

\begin{lemma}\label{lemma:main_variant}
Suppose that $\BL({\mathbf{B},\mathbf{p}})$ is finite, that $(\mathbf{B},\mathbf{p})$ is 
simple, and that $p_j\in(1,\infty)$ for every index $j\in\{1,2,\dots,m\}$.
For every $\eps>0$ there exists $\delta>0$ such that 
for any tuple $\bff = (f_j: 1\le j\le m)$ of positive Gaussians satisfying
$\blbp{\bff} \ge (1-\delta)\BLBp$,
there exists a tuple $\bg = (g_j: 1\le j\le m)$ of 
positive Gaussians that 
realizes the optimal constant $\BLBp$ 
 and satisfies
\[ \norm{f_j-g_j}_{L^{p_j}}
\le \eps\|f_j\|_{L^{p_j}(\mathbb{R}^{d_j})}\ \forall\,j\in\{1,2,\dots,m\}.  \]
\end{lemma}

\begin{proof}
Let $\Bp$ be simple. 
Then as shown in \cite{BCCT}, matters can be reduced by 
linear changes of variables in $\reals^d$
and $\reals^{d_j}$ to the geometric case, in which $B_jB_j^*$ equals
the identity operator on $\reals^{d_j}$ for every index $j$, 
and $\sum_j p_j^{-1} B_j^*B_j$ is the identity on $\reals^d$.

We begin with the centered Gaussian case.
Denote by $\langle\,\cdot\,\rangle$ the Euclidean inner product on $\reals^{n}$ 
for arbitrary $n$.
By a symmetric linear map $A:\reals^n\to\reals^n$
we mean one that satisfies
$\langle A_j y,\,y'\rangle = \langle y,\,A_j y'\rangle$ for all $y,y'\in\reals^n$.

For each index $j$ let $A_j:\reals^{d_j}\to\reals^{d_j}$ 
be a positive definite symmetric linear transformation.
The set of all such tuples $\bA = (A_j: j\in\{1,2,\dots,m\})$
can be regarded a subset of a Euclidean space by using the matrix
entries, with any fixed orthonormal bases for $\reals^{d_j}$,
as coordinates. We regard the set of all such $\bA$ as a topological
space, under the relative topology induced by this embedding.
We denote by $\det(A)$ the determinant of the matrix representing $A$
with respect to the relevant orthonormal basis.

Consider any input tuple $\bff = (f_j)$ of the centered form
\begin{equation} \label{cpGform}
f_j(y) = e^{-\langle A_j y,\,y\rangle}\ \forall\,y\in\reals^{d_j} \ \forall\,j,
\end{equation}
with $A_j$ symmetric and positive definite.
Write the Brascamp--Lieb inequality in the equivalent form
\eqref{eq:BLlinearformcj}, with $q_j = 1/p_j$.
Thus $f_j(B_j)^{q_j}(x) = e^{-\langle q_j B_j^* A_j B_j x,\,x\rangle}$ and 
\[ \prod_{j=1}^m f_j^{q_j}(B_j(x)) = e^{-\langle M_\bA x,\,x\rangle}\]
where 
\begin{equation} \label{MbA} M_\bA = \sum_{j=1}^m q_j B_j^* A_j B_j\end{equation}
is a positive semidefinite symmetric $d\times d$ matrix.

Finiteness of $\BLBp$ implies that $\bigcap_{j=1}^m \nullspace(B_j)=\{0\}$,
so the matrix $M_\bA$ defined in \eqref{MbA} is positive definite,
rather than merely semidefinite.  
Define a topological space $\scriptt$ to be the set 
of all tuples $\bA = (A_j: 1\le j\le m)$ of positive definite symmetric matrices
satisfying $\det(M_\bA)=1$, equipped with its relative topology. 

When $\bff$ takes the centered positive Gaussian form \eqref{cpGform},
we also denote the quantity $\BL(\bB,\bp; \bff)$ by $\BL(\bB,\bp; \bA)$.
We will exploit the expression
\begin{equation} \label{blbpa_expression} 
\BL(\bB,\bp; \bA) = \det(M_\bA)^{-1/2} \prod_j \det(A_j)^{q_j/2},\end{equation}
which simplifies to $\prod_j \det(A_j)^{q_j/2}$ for $\bA\in\scriptt$.

For $r\in(0,\infty)$ we write $r\bA = (rA_j)$. 
The quantity $\BL(\bB,\bp; \bA)$ enjoys the scaling symmetry
\begin{equation} \label{scalingsymmetry} \BL(\bB,\bp; \bA) = \BL(\bB,\bp;r\bA) \ 
\text{ for all $r\in(0,\infty)$ and all $\bA$.}\end{equation} 
This relation is a consequence of the scaling condition 
\eqref{blbpa_expression}  together with \eqref{BLscaling}, 
which is necessary for the finiteness of $\BLBp$.
Since $\BL(\bB,\bp)$ is simple, there exists a centered Gaussian maximizer, 
and it is unique up to this scaling symmetry by Corollary~9.2 of \cite{BCCT}.
Therefore there exists a unique centered Gaussian maximizer $\bAstar$
satisfying $\det(M_{\bAstar})=1$.

\begin{claim*}
For each $\eta>0$, the set $\scriptt_\eta$ of all tuples
$\bA$ of positive definite symmetric matrices
satisfying $\BL(\bB,\bp; \bA) \ge \eta$ is a compact subset of $\scriptt$.
\end{claim*}

\begin{proof}[Proof of claim] 
Consider any $\bA\in\scriptt$ and
denote by $\lambda_1\ge\lambda_2\ge\cdots \ge \lambda_d>0$ the eigenvalues of $M_\bA$,
in descending order.
It is shown in the proof of Proposition~5.2 of \cite{BCCT} that
there exist $r>0$ and $C<\infty$, depending on $\Bp$, such that for any $\bA$,
\[ \prod_{j=1}^m (\det(A_j))^{q_j} \le C\det(M_\bA) (\lambda_{d}/\lambda_1)^r.\] 
Thus for any $\delta>0$, $ \BL(\bB,\bp; \bA) \le C^{1/2}\delta^{r/2}$
unless $\lambda_d/\lambda_1\ge\delta$.
From this conclusion together
with the normalization $\det(M_\bA)=1$ it follows that
for each $\eta>0$ there exists $C'<\infty$
such that for every $\bA\in\scriptt_\eta$,
the largest eigenvalue of $M_\bA$ satisfies $\lambda_1\le C'$. 

Let $\eta>0$ be arbitrary, and consider any $\bA\in\scriptt_\eta$.
Since each $B_jB_j^*$ equals the identity,
$\trace(M_\bA) = \sum_j q_j \trace(A_j)$. 
Since all eigenvalues of $A_j$ are nonnegative,
each eigenvalue of $A_j$ is consequently
majorized by $dq_j^{-1}\lambda_1\le dq_j^{-1}C'<\infty$. 
Therefore there exists $C''<\infty$ such that for any $\bA\in\scriptt_\eta$ and any index $j$, 
all eigenvalues of $A_j$ lie in $(0,C'']$. 
Therefore $\det(A_j)$ is likewise majorized by a positive constant. 
By \eqref{blbpa_expression} 
this implies that $\BL(\bB,\bp; \bA) = O(\max_j \det(A_j)^{q_j/2})$
uniformly for all $\bA\in\scriptt_\eta$.
This implies a strictly positive uniform lower bound for $\det(A_j)$
uniformly for all $\bA\in\scriptt_\eta$
and again, this together with the upper bound for all eigenvalues
of $A_j$ implies a strictly positive uniform lower bound
for all eigenvalues of $A_j$.
From the compactness of the space of all possible orthonormal frames for $\reals^d$
it now follows that $\scriptt_\eta$ is compact.
\end{proof}

Choose $0<\eta<\BL(\bB,\bp)$,
so that the unique maximizer $\bAstar$
satisfying $\det(M_{\bAstar})=1$ belongs to $\scriptt_\eta$.
%Since the mapping $\scriptt\owns\bA\mapsto \BL(\bB,\bp;\bA)$ is continuous
%and $\scriptt_\eta$ is compact, there exists $\bAstar\in \scriptt_\eta$
%satisfying $\BL(\bB,\bp;\bAstar) = \BL(\bB,\bp)$. 
%Since $(\bB,\bp)$ is simple, a solution $\bAstar$ of this 
%equation is unique up to the scaling symmetry \cite{scaling}, 
%hence by the normalization $\det(M_\bA)=1$ is fully unique.
%$\bAstar$ is unique by Corollary~9.2 of \cite{BCCT}. 
Since $\scriptt_\eta$ is compact,
since the function $\bA\mapsto \BL(\bB,\bp; \bA)$ is continuous,
and since $\bAstar$ is the unique maximizer in $\scriptt_\eta$,
%$\BL(\bB,\bp;\bA)\le\eta<\BL(\bB,\bp)$ for all $\bA\notin\scriptt_\eta$,
for each $\eps>0$ there exists $\delta>0$ such that any $\bA\in\scriptt_\eta$ satisfying
$\BL(\bB,\bp; \bA)>(1-\delta)\BL(\bB,\bp)$
will satisfy $|\bA-\bAstar|<\eps$.
If $\delta$ is chosen to also satisfy $(1-\delta) \BL(\bB,\bp)>\eta$
then no $\bA\in\scriptt\setminus\scriptt_\eta$ 
satisfies $\BL(\bB,\bp; \bA)>(1-\delta)\BL(\bB,\bp)$.
By the scaling symmetry \eqref{scalingsymmetry}, 
this completes the analysis of centered Gaussian inputs.

Next, consider general nonnegative Gaussian inputs, not necessarily centered. 
Define $V \subset \reals^{d_1}\times\cdots\times\reals^{d_m}$
to be the linear subspace consisting of all $\bv = (v_j: j\in\{1,2,\dots,m\})$
for which there exists $\barx\in\reals^d$ satisfying $v_j = B_j(\barx)\ \forall\,j$. 

Any nonnegative Gaussian input can be expressed in the form
\begin{equation} \label{generalG} 
f_j(y) = c_j e^{-\langle A_j(y-v_j),\,y-v_j\rangle}
\end{equation}
with $c_j>0$, $v_j\in\reals^{d_j}$, and $A_j$ symmetric and positive definite.
Both the quantity $\BL(\bB,\bp; \bff)$
and the projective distances in the inequality
are invariant under multiplication of $f_j$ by positive constants,
depending on $j$, and consequently we may assume with no loss of generality
that $c_j=1$ for every $j$.
By completing the square, 
$F(x) = \prod_{j=1}^m f_j(B_j x)^{q_j}$ may then be expressed in the form 
\begin{equation} \label{c_defn}
F(x) = c e^{-\langle M_\bA(x-\barx),(x-\barx)\rangle}
\end{equation}
where  
\begin{equation*}
\barx =M_\bA^{-1}\sum_{j=1}^m q_jB_j^*A_jv_j
\ \ \text{ and } \ \ 
c = \prod_{j=1}^m e^{- q_j\langle A_j(B_j(\barx)-v_j),\,B_j(\barx)-v_j\rangle} \le 1.
\end{equation*}

Upon defining
$g_j(y) = e^{-\langle A_j y,\,y\rangle}$ 
we have
\begin{equation} \label{comparison}
\BL(\bB,\bp; \bff) = c\BL(\bB,\bp; \bg).
\end{equation}
Therefore if $\bff$ is an extremizer then $c=1$. 
This occurs if and only if each factor
in the product expression for $c$ equals $1$, that is, if and only if 
$v_j = B_j(\barx)$ for each $j$, whence $\bv\in V$.

Moreover, by \eqref{comparison}, $\bff$ is an extremizer if and only if
$c=1$ and $\bg$ is an extremizer.
Thus a nonnegative Gaussian input $\bff$ of the general form \eqref{generalG}
is an extremizer if and only if $\bA$ is a scalar multiple of $\bAstar$
and $\bv\in V$.

\medskip
Now let $\eps>0$ be arbitrary, and consider any positive Gaussian datum
$\bff$ of the form \eqref{generalG}
that satisfies $\BL(\bB,\bp;\bff)\ge (1-\eps) \BLBp$.
Let $\bg$ be the centered Gaussian  datum associated to $\bff$,
as above.
Also let $\bg_*$ be a nonnegative centered Gaussian input satisfying
$\BL(\bB,\bp;\bg_*) = \BLBp$, with associated $A_{*,j}$  
and $\bAstar$.
By the scaling symmetry, we may assume that $\det(M_\bA) = \det(M_{\bAstar})=1$.

By \eqref{comparison}, 
\[ \BL(\bB,\bp;\bg)\ge 
\BL(\bB,\bp;\bff)\ge 
(1-\eps) \BLBp.\]
Therefore by applying the result for centered Gaussian data,
proved above, to $\bg$ we may conclude that
$\|A_j-A_{*,j}\|<\delta_1(\eps)$ for each index $j$, where $\delta_1(\eps)\to 0$ as $\eps\to 0$.

In particular, if $\eps$ is sufficiently small, as we may assume, then
$\langle A_j y,y\rangle \ge \tfrac12 \langle A_{*,j}y,y\rangle$,
for every $y\in\reals^{d_j}$, for each index $j$.
Now
\[\BL(\bB,\bp;\bff) \le c \BL(\bB,\bp;\bg) \le c \BLBp,\]
where $c$ is defined by \eqref{c_defn}
for a certain element $\barx\in\reals^d$ and hence satisfies
\[ \ln(c) \le -\tfrac12 
\sum_{j=1}^m q_j \langle A_{*,j}(B_j(\barx)-v_j),\,B_j(\barx)-v_j\rangle.\]
The positive quantity
$\sum_{j=1}^m q_j \langle A_{*,j}(B_j(\barx)-v_j),\,B_j(\barx)-v_j\rangle$
is bounded below by a constant multiple of
the square of the distance from $\bv$ to $V$,
with constant factor depending only on $(A_{*,j})$. 
All that remains is to appeal to the continuity of the map 
$(A,v,c)\mapsto ce^{-\langle A(\cdot-v),\cdot-v\rangle}:\symm_+(\mathbb{R}^n)\times\mathbb{R}^n\times\mathbb{R}_+\rightarrow L^p(\mathbb{R}^n)$. 
\end{proof}

\end{document}